\documentclass[reqno, 11pt, twoside]{amsart}

\usepackage{amsfonts, amsthm, amsmath, amssymb}
\usepackage{hyperref}
\hypersetup{colorlinks=false}

\usepackage[margin=3.5cm]{geometry}

\RequirePackage{mathrsfs} \let\mathcal\mathscr

\renewcommand{\leq}{\leqslant}

\renewcommand{\geq}{\geqslant}

\renewcommand{\d}{\mathrm{d}}

\usepackage{comment}
\usepackage{graphics}
\usepackage{aliascnt}

\usepackage{enumerate}
\usepackage{amsmath}
\usepackage{amsfonts}
\usepackage{amssymb}
\usepackage{amsthm}
\usepackage{comment}
\usepackage{mathtools}
\usepackage{xcolor}

\newtheorem{theorem}{Theorem}[section]
\newtheorem{corollary}[theorem]{Corollary}

\newtheorem{lemma}[theorem]{Lemma}
\newtheorem{proposition}[theorem]{Proposition}
\theoremstyle{definition}

\newtheorem{remark}[theorem]{Remark}
\newtheorem{conjecture}[theorem]{Conjecture}
\numberwithin{equation}{section}

\newcommand\Z{\mathbb{Z}}
\newcommand\ZZ{\mathbb{Z}}
\newcommand\N{\mathbb{N}}
\newcommand\NN{\mathbb{N}}

\newcommand\CC{\mathbb{C}}

\newcommand\QQ{\mathbb{Q}}

\newcommand{\f}{\mathbf{f}}

\DeclareMathOperator{\rad}{rad}

\DeclareMathOperator{\n}{N}

\usepackage{filecontents}

\usepackage[backend=biber,
style=numeric,
isbn=false,
doi=false
]{biblatex}
\title{Square-free values of polynomials on average}

\author{Pascal Jelinek}
\address{Universität Wien\\
		Oskar-Morgensternplatz 1\\
		1100 Wien\\
		Austria}
\email{pascal.jelinek@univie.ac.at}

\subjclass[2010]{11N32}

\begin{document}
    \begin{abstract}
        The number of square-free integers in $x$ consecutive values of any polynomial $f$ is conjectured to be $c_fx$, where the constant $c_f$ depends only on the polynomial $f$. This has been proven for degrees less or equal to 3. Granville was able to show conditionally on the $abc$-conjecture that this conjecture is true for polynomials of arbitrarily large degrees. In 2013 Shparlinski proved that this conjecture holds on average over all polynomials of a fixed naive height, which was improved by Browning and Shparlinski in 2023.
        In this paper, we improve the dependence between $x$ and the height of the polynomial. We achieve this via adapting a method introduced in a 2022 paper by Browning, Sofos, and Teräväinen on the Bateman-Horn conjecture, the polynomial Chowla conjecture, and the Hasse principle on average.
    \end{abstract}
    
\maketitle

\thispagestyle{empty}
\setcounter{tocdepth}{1}
\tableofcontents
\section{Introduction}
The number of square-free values of a given polynomial has been studied for more than a century, with the first results dating back at least to Landau \cite{landau_handbuch_1911} in 1911. He showed that infinitely many square-free values are attained by any linear univariate polynomial with square-free content.\\
More generally, we have the following heuristic: Given a polynomial $f$, let $c_p$ be the number of solutions of $f$ when viewed over $\ZZ/p^2\ZZ$, i.e., $c_p=\#\{(a_1,\dots,a_n)\in(\ZZ\cap[1,p^2])^n:f(a_1,\dots,a_n)\equiv 0 \bmod p^2\}$. If we now assume that the primes are independent of each other, we get the following conjecture:
\begin{conjecture}\label{Conj_General}
	Let $f\in\ZZ[x_1,\dots,x_n]$. We define $S:=\{(a_1,\dots,a_n)\in\ZZ^n:f(a_1,\dots,a_n) \text{ is square-free}\}$. Then for any $\text{Box}=\text{Box}(N_1,\dots,N_n)=\{(a_1,\dots,a_n)\in\ZZ^n:|a_i|<N_i \text{ for all }i\}$ we have that
	\begin{equation}
		\lim_{N_1, \dots,N_n\to \infty} \frac{\#(S\cap\text{Box})}{\#\text{Box}}=\prod_{p}\left(1-\frac{c_p}{p^{2n}}\right).
	\end{equation}
\end{conjecture}
A consequence of this conjecture is that the polynomial $f$ attains infinitely many square-free values if it is not the square of another polynomial $g$ and if it has non-zero values mod $p^2$ for each prime number $p$.
\begin{remark}
	In the case where $f$ is a univariate polynomial, we will often use the term $\rho_f(p^2)$ instead of $c_p$.
\end{remark}
In 2002, Poonen \cite{Poonen2002SquarefreeVO} showed that this problem regarding a general polynomial $f\in \ZZ[x_1,\dots,x_n]$ can be reduced to the analogous problem concerning univariate polynomials. He achieved this via considering $f$ as a polynomial in $\ZZ[x_1\dots,x_{n-1}][x_n]$. He then combined this reduction with prior work by Granville \cite{Granville1998ABCAU}, who showed that Conjecture \ref{Conj_General} holds for univariate polynomials under the $abc$-conjecture. Poonen emphasises that this method only provides a proof to a slightly weaker form of the conjecture (he needed to impose some relations between the $N_i$'s) since the proof of Granville under the $abc$-conjecture is not homogeneous in the coefficients. However, the relations between the $N_i$'s could be dropped with a new proof of the univariate case. (cf. remark after Lemma 6.2 in \cite{Poonen2002SquarefreeVO})\\
If we do not assume the $abc$-conjecture, then the conjecture is wide open, both regarding multivariate polynomials and regarding univariate polynomials. In both cases, it is solved only for specific families of polynomials, of which we now want to highlight a few.
\subsection{Results regarding multivariate polynomials}
If the degree is low compared to the number of variables, one can use the circle method to show that the conjectured asymptotical behaviour is true. Remarkably, work by Destagnol and Sofos \cite{DESTAGNOL2019102794} shows that in this case, a third of the number of variables compared to the Birch setting already suffices. Explicitly, this means that if $n-\sigma_f>1/3(d-1)2^d$, the conjecture holds, where $n$ is the number of variables, $d$ is the degree of the polynomial $f$ and $\sigma_f$ is the degree of the singular locus of $f=0$.\\
Another family of polynomials for which it has been shown that the heuristic holds is due to work of Bhargava \cite{bhargava2014geometric}. If $f$ is invariant under the action of a suitably large algebraic group, then the heuristic correctly predicts the asymptotic behaviour of square-free numbers. In particular, the polynomials related to the discriminants of degree 3, degree 4, and degree 5 extensions of $\QQ$ satisfy this criterion.\\
These results and their methods of proof were also used to prove lower bounds on the average size of Selmer groups of families of elliptic curves and also to prove that the Hasse principle fails for a positive proportion of plane cubic curves over $\QQ$. (see for example \cite{bhargava2014positive}, \cite{bhargava2022average}, \cite{bhargava2013binary}, \cite{bhargava2013ternary}, \cite{bhargava2013average4} and \cite{bhargava2013average5})
\subsection{Results regarding univariate polynomials}
In the case of a univariate polynomial $f\in\ZZ[t]$, we will state the conjecture in the following way:
\begin{conjecture}
	Let $f\in\ZZ[t]$. Then:
	\begin{equation}
		\sum_{1\leq n\leq x}\mu^2(f(n))\sim\prod_{p}\left(1-\frac{\rho_f(p^2)}{p^2}\right)x,
	\end{equation}
	where $\rho_f\left(m\right)$ is defined to be the number of solutions to $f\left(n\right)\equiv 0 \bmod m$, where $1\leq n\leq m$.
\end{conjecture}
There have been many different results related to this conjecture, and we want to highlight a few of them:\\
Ricci \cite{Ricci1933RicercheAS} was able to show that the asymptotics holds for polynomials of degree $\leq 2$. If $f$ has degree 1, then an explicit error term and an explicit dependence between $x$ and the height of the polynomial were first stated by Prachar \cite{Prachar1958berDK}. The error term has subsequently been improved by Hooley \cite{hooley_1975} and the dependence between $x$ and the height of the polynomial has been improved by Nunes \cite{Nunes_2017}.\\
If $f$ has degree 3, the asymptotics was first proven by Erd\H{o}s \cite{Erdös_Cube}. Later Hooley \cite{Hooley1967OnTP} provided the first explicit log-saving error term, which Reuss \cite{Reuss2013PowerfreeVO} improved to a power-saving error term in 2013.\\
Erd\H{o}s also conjectured that the inputs can be restricted to the prime numbers, in which case it has been shown by Helfgott \cite{Helfgott2007PowerfreeVR} to hold up to degree 3.\\
It appears that for univariate polynomials of degree $>3$, there is no polynomial known that attains infinitely many square-free values. \\
\subsection{Averaging and announcement of new results}
In the last few years, there has been much progress on the average behaviour of arithmetic functions over polynomial values, among others behaviour related to the van Mangold function $\Lambda$, Liouvilles function $\lambda$ and the $r$-function.\\
Regarding square-free numbers, we want to study the following problem:\\
Consider a polynomial $g=c_0t^d+\dots+c_d\in \ZZ[t]$ of height at most $H$ and degree at most $d$. Let $S\subset\{0,\dots,d\}$. We associate each polynomial to its coefficient vector and define $\mathcal{F}_{g,S}(H)$ to be roughly the set of coefficient vectors whose entries agree with or are close to the entries of $g$, depending on whether the index of the coefficient is in $S$ or not. More precisely, $\mathcal{F}_{g,S}(H)$ is defined to be
$$
\left\{(a_0,\dots,a_d) \in (\ZZ)^{d+1}: 
\begin{array}{ll}
	\text{$a_i=c_i$}& \text{if $i\notin S$}\\
	\text{$|a_i-c_i|\leq H$}& \text{if $i\in S$}\\
	\gcd(a_0,\dots,a_d)=1&
\end{array}
\right\}.
$$
We want to show that as $H$ tends to infinity, almost all $f\in\mathcal{F}_{g,S}(H)$ satisfy the conjecture. This is achieved via proving that
\begin{equation}\label{method}
	\frac{1}{\#\mathcal{F}_{g,S}(H)}\sum_{f\mathcal{F}_{g,S}(H)}\left|\sum_{1\leq n\leq x}\mu^2(f(n))-c_fx\right|^k\ll x^{1-\delta}
\end{equation}
for some exponent $k$ and $c_f=\prod\limits_{p}\left(1-\frac{\rho_f\left(p^2\right)}{p^2}\right)$.\\
All cited results below were derived using $k=1$.\\
Shparlinski \cite{shparlinski_average_2013} showed that given $g\equiv 0$ and $S=\{0,\dots,d\}$, the conjecture holds for almost all polynomials under the condition that $$x^{d-1+\varepsilon}\leq H\leq x^A,$$ for some constant $A$. In recent work, Browning and Shparlinski \cite{browning2023squarefree} used the geometry of numbers and the determinant method to improve this result to $$x^{d-3+\varepsilon}\leq H\leq x^A.$$
In the same paper, Browning and Shparlinski showed that if $g=c_0x^{d}+\dots+c_{d-1}x^{1}$ and only the constant coefficient is varied, i.e. $S=\{d\}$, then the dependence between $x$ and $H$ is
$$x^{(d-1)/2+\eta(d)+\varepsilon},$$ where $$\eta(d)=\begin{cases}
	2^{-(d+1)} & \text{if } 2\leq d\leq 5\\
	\frac{1}{d(d-1)} & \text{if } d\geq 6.
\end{cases}$$
This was achieved via adapting some works on the Vinogradov mean value theorem. (see \cite{browning2023squarefree} and the references therein)\\
In this paper, we approach this problem by considering inequality \eqref{method} for k=2. We adapt 
methods developed in recent work by Browning, Sofos and Teräväinen \cite{browning2022batemanhorn} to prove the following theorem:
\begin{theorem}\label{thm_aim}
	Let $A, d>1$, $1/5>\varepsilon>0$ be fixed, $d\geq\ell>k\geq0$. Then there is a constant $H_0(A,d)$ such that for all $H>H_0(A,d)$ the following statement holds: Let $g$ be any polynomial of degree $\leq d$ with $g(0)\neq0$ and let $S=\{\ell,k\}$. Let $x$ be such that
	\begin{enumerate}
		\item $x^{2\ell-d}\leq H$;
		\item $x^{d+\ell-2k+14/5+\varepsilon}\leq H\leq x^A$.
	\end{enumerate}
	Then there exists some $\delta>0$ such that for all but 
	$\ll H^{2-\delta}$ many
	degree $d$ polynomials $f$ in $\mathcal{F}_{g,S}(H)$, where the implied constant only depends on fixed parameters, we have
	\begin{equation}
		\left|\frac{1}{x}\sum_{1\leq n\leq x} \mu^2(f(n)) - c_fx\right|\leq x^{-\delta}.
	\end{equation}
\end{theorem}
An explicit value that can be chosen for $\delta$ will be presented in the proof in section \ref{Chap_5}.
\begin{remark}
	By going through the proof of Theorem \ref{thm_aim} above, one can see that under the condition that $\ell/2+k>d+7/5$, the result above still holds, even if $x^{2\ell-d}\geq H$, after a slight change in the second relation between $x$ and $H$.
\end{remark}
\begin{remark}
	We can achieve the minimal error term by taking $\ell=1$ and $k=0$. Using the notation from above, we have the following statement:\\
	If $x$ is such that
	$$
	x^{d+19/5+\varepsilon}\leq H\leq x^A,
	$$
	then for all but
	$$
	H^{2-\frac{1-\varepsilon}{(5d+19)(2+d)}}
	$$
	degree $d$ polynomials $f$ in $\mathcal{F}_{g,S}(H)$, where the implied constant only depends on fixed parameters, we have
	\begin{equation}
		\left|\frac{1}{x}\sum_{1\leq n\leq x} \mu^2(f(n)) - c_fx\right|\leq x^{-\frac{1-\varepsilon}{20+10d}}.
	\end{equation}
\end{remark}
\begin{remark}
	We can achieve the maximal range of values of $x$ by choosing $\ell=\left[2d/3\right]$. Let $r(d)=2d/3-\ell$. Then using the notation from above, we have, after optimising for $k$, the following statement:\\
	Let $k=\ell-1$. If $x$ is such that
	$$
	x^{d/3+4/5+r(d)+\varepsilon}\leq H\leq x^A,
	$$
	then there exists some $\delta>0$ such that for all but
	$
	\ll H^{2-\delta}
	$ many
	degree $d$ polynomials $f$ in $\mathcal{F}_{g,S}(H)$, where the implied constant only depends on fixed parameters, we have
	\begin{equation}
		\left|\frac{1}{x}\sum_{1\leq n\leq x} \mu^2(f(n)) - c_fx\right|\leq x^{-\delta}.
	\end{equation}
\end{remark}

\subsection{Structure of the paper}
In Section \ref{Chap_2}, we present the tool that allows us to prove the result on average for random polynomials via equidistribution in short intervals and in arithmetic progressions. In Section \ref{Chap_3}, we derive results on how $\mu^2(\cdot)$ is distributed in the relevant cases. In Section \ref{Chap_4}, we find an approximation of $\mu^2(\cdot)$ in short intervals and in arithmetic progressions, which enables us to apply Corollary \ref{cor_general} to the setting of Theorem \ref{thm_aim} and to prove the theorem in Section \ref{Chap_5}.
\subsection{Notation}
$h$, $i$, $j$, $k$, $\ell$ and $q$ will always be integers and $p$ will always denote a prime number. Further, $I$ will always be an interval and $\mu(\cdot)$ will always denote the Möbius function, $\varphi(\cdot)$ the Euler totient function and $\tau(\cdot)$ the divisor function.\\
$A$, $d$ and $\varepsilon$ will be constants, where $A$ is usually very large, and $\varepsilon$ very small.
We will adopt Vinogradov's notation $O(\cdot)$ and $\ll$, and we will allow the implied constant to depend only on constant factors $A$, $d$ and $\varepsilon$ unless indicated in the subscript.

\subsection{Acknowledgements}
This paper is the result of my Master's thesis at the University of Vienna, written under the supervision of Tim Browning.\\
I am greatly indebted to Tim Browning, who welcomed me into his research group without hesitation and without whose insights and tips this project would not have been possible.\\
Special thanks go also to Igor Shparlinski for his valuable feedback.

\section{Main tool}\label{Chap_2}
We present a corollary of Theorem 2.2 in \cite{browning_bateman-horn_2022}. In the setting of this thesis, we expect a power-saving instead of a log-saving, so we change the second assumption accordingly. 
\begin{corollary}
	\label{cor_general} Let $A\geq 1$, $d>1$, $\delta>\varepsilon>0$ and $0\leq k< \ell \leq d$ be fixed. Let $H\geq H_0(A,d)$ and $x^{\ell-k+\varepsilon}<H\leq x^A$. Let $F:\mathbb{Z}\to \mathbb{C}$ be a sequence such that 
	\begin{enumerate}
		\item $|F(n)|\ll n^{\varepsilon}$ 
		for all $n \in \mathbb Z$;
		\item For any $q\leq x^{\ell}$
		\begin{align*}
			\max_{\substack{
					1\leq u\leq q\\
					\gcd(u,q)\leq x
			}}  
			\sup_{\substack{
					I  \textnormal{ interval } 
					\\
					|I|> H^{1-\varepsilon}x^k
					\\
					I\subset [-2H x^d, 2H x^d]
			}}
			\frac{q}{|I|}
			\left|\sum_{\substack{n\in I\\n\equiv u\bmod q}}F(n)\right|\ll x^{-\delta}.
		\end{align*}
	\end{enumerate} 
	Then, for any polynomial $g\in \mathbb{Z}[t]$ of degree $\leq d$ with coefficients in $[-H,H]$ and $g(0)\neq0$, and for any coefficients $\alpha_n\in \CC$ such that $|\alpha_n|\leq 1$, 
	we have 
	\begin{align}\label{equ2}
		\sup_{x'\in [x/2,x]}\sum_{|a|,|b|\leq H}\left|\sum_{1\leq n\leq x'}\alpha_nF(an^k+bn^{\ell}+g(n))\right|^2\ll \frac{H^2x^{2}}{x^{\delta'}},
	\end{align}
	with
	\begin{align*}
		\delta'
		=\min\left(\frac{\delta-\varepsilon}{1+\ell+d(\ell-k)} ,\frac{1-\varepsilon}{1+3\ell-2k+d}\right).
	\end{align*}
	
\end{corollary}

To prove this corollary, we proceed along the same lines as in the proof of Theorem 2.2 in \cite{browning_bateman-horn_2022}. We first state some lemmata which we use in the proof. An analogous version of each lemma can be found in \cite{browning_bateman-horn_2022}.

\begin{lemma}\label{large_div}
	Let $q\in\NN$ and $z\geq 1$. If $q\leq z^A$, then
	\begin{align}
		\sum_{\substack{d\mid q \\ d\geq z}}\frac{1}{d}\ll\frac{1}{z^{1-\varepsilon}}.
	\end{align}
\end{lemma}
\begin{proof}
	By the standard estimate for the divisor function and by the assumption $q\leq z^A$, we obtain that
	$$
	\sum_{\substack{d\mid q \\ d\geq z}}\frac{1}{d}\ll
	\frac{1}{z}\sum_{d\mid q}1\ll
	\frac{1}{z^{1-\varepsilon}}.
	$$
\end{proof}
\begin{lemma}\label{large_gcd}
	Let $A\geq1$, $a,c \in \ZZ\backslash\{0\}$, $b\in \ZZ$ and $x_1,x_2\geq 1$, such that $|c|<x_2^A$. Then we have
	$$
	\#\{n\in\ZZ\cap[-x_1,x_1]:\gcd(an+b,c)>x_2\}\ll \frac{x_1\gcd(a,c)}{x_2^{1-\varepsilon}}+\tau(c).
	$$
\end{lemma}
The proof follows the proof in \cite{browning_bateman-horn_2022}, now using \eqref{large_div} as an estimate.

\begin{lemma}\label{le_typical} Let $\varepsilon>0 $, $d\in \N$, $d>\kappa>0$ be fixed and let $x> 1 $. Let $g\in\ZZ [t] $ be a polynomial of degree $\leq d$ such that $g(0)\neq0$. Define
	\begin{equation}\label{eq:define_M}
		\mathcal{M}_{\kappa,d}:=\left\{\mathbf n \in (\N\cap [1,x])^2: 
		\begin{array}{l}
			\text{$n_1,n_2>x^{1-\kappa/d}$ and $|n_1-n_2|>x^{1-\kappa/d}$}\\
			\gcd(n_1,n_2)<x^{\kappa/d}\\
			\gcd(n_1^d,g(n_1))<x^{\kappa+\varepsilon}
		\end{array}
		\right\}.
	\end{equation}
	Then $\#((\N\cap[1,x])^2\setminus \mathcal{M}_{\kappa,d})\ll x^{2-\kappa/d}$.
\end{lemma}

We show Lemma \ref{le_typical} completely analogously to the corresponding statement in \cite{browning_bateman-horn_2022}. In both cases, the proof is done by considering the number of pairs failing each property individually and then applying the union bound.\\
Now we have stated all the necessary lemmata to prove the corollary.
\begin{remark}
	Due to the assumption that $x^{\ell-k+\varepsilon}<H\leq x^A$, we may interchange $(Hx^d)^{\varepsilon}$, $ H^{\varepsilon}$ and $x^{\varepsilon}$ throughout the proof.
\end{remark}
\subsection{Proof of Corollary \ref{cor_general}}
In the proof, we will expand the left-hand side of \eqref{equ2}, and we will show that the inequality holds for all $x'\in\left[x/2,x\right]$ and $g\in\ZZ\left[t\right]$, hence also for the maximal value that any such combination of $x'$ and $g$ can reach.\\
The left-hand side now looks like:
\begin{align}\label{eq_start}
	\sum_{1\leq n_1,n_2\leq x'} \alpha_{n_1}\overline\alpha_{n_2}\sum_{|a|,|b|\leq H}F(an_1^k+bn_1^{\ell}+g(n_1))\overline{F}(an_2^k+bn_2^{\ell}+g(n_2)).  
\end{align}
Using the notation of Lemma \ref{le_typical}, we define the sets $\mathcal{M}:=\mathcal{M}_{\kappa,d}$ and $\mathcal{M}^c := \left\{(n_1,n_2)\in \NN^2:n_1,n_2 \leq x\right\}\backslash\mathcal{M}$, where $\kappa$ is a parameter that will be defined later. In order to consider their contributions individually, we define the following two cases:
\begin{enumerate}
	\item $(n_1,n_2)\in \mathcal{M}^c$;
	\item $(n_1,n_2)\in \mathcal{M}$.
\end{enumerate}
Lemma \ref{le_typical} gives us that
\begin{equation}
	\#\mathcal{M}^c\ll x^{2-\kappa/d}.
\end{equation}
\subsubsection{Contribution of $(n_1,n_2)\in \mathcal{M}^c$}
We can use the estimate above to bound the contribution of $\mathcal{M}^c$. By the first assumption of Corollary \ref{cor_general}, we have that
\begin{align*}
	&\sum_{n_1,n_2\in \mathcal{M}^c} \alpha_{n_1}\overline\alpha_{n_2}\sum_{|a|,|b|\leq H}F(an_1^k+bn_1^{\ell}+g(n_1))\overline{F}(an_2^k+bn_2^{\ell}+g(n_2))\\
	&\ll\# \mathcal{M}^c H^2H^{\varepsilon}\\
	&\ll H^2x^{2-\kappa/d+\varepsilon}.
\end{align*}
We now turn to Case $(2)$ and follow the argument in \cite{browning_bateman-horn_2022}.
\subsubsection{Contribution of $(n_1,n_2)\in \mathcal{M}$}
We first take absolute values and use the triangle inequality to bound the contribution of $\mathcal{ M}$ to \eqref{eq_start}. Hence, the contribution is
\begin{equation}\label{eq_start2}
	\ll\sum_{\mathbf{n} \in \mathcal{M}} \left|\alpha_{n_1}\overline\alpha_{n_2}\right|\left|\sum_{|a||b|\leq H}F(an_1^k+bn_1^{\ell}+g(n_1))\overline{F}(an_2^k+bn_2^{\ell}+g(n_2))\right|. 
\end{equation}
Since $\left|\alpha_{n_1}\right|,\left|\alpha_{n_2}\right|\leq 1$, and after using the triangle inequality again, \eqref{eq_start2} can be written as
\begin{align}\label{eq:9equ10} 
	\ll\sum_{\mathbf n \in \mathcal{M}}
	\sum_{m_2\in \mathbb{Z}}F(n_2^km_2+g(n_2))\left|\sum_{m_1\in \mathbb{Z}}
	F(n_1^{k}m_1+g(n_1))
	\gamma(\mathbf m ) \right|, 
\end{align}
where we take $m_i=a+bn_i^{\ell-k}$ for $i=1,2$ and let
$$ \gamma(\mathbf m ) :=\#\left\{(a,b) \in (\Z\cap [-H,H] )^2: 
m_i=a +bn_i^{\ell-k} \ \forall i=1,2 \right\}.
$$
From now on, we will use the following shorthand:
$$
\Delta :=n_1^{\ell-k}-n_2^{\ell-k}.
$$
This can be bounded below using the first property of $\mathcal{M}$:
\begin{align}\label{eqq17}
	|\Delta|\geq |n_1 -n_2| n_1^{\ell-k-1}
	\gg x^{(\ell-k)(1-\kappa/d)}.    
\end{align} 
We have by the definition of $m_2$ that
$$ 
| m_2| \leq H +n_2^{\ell-k} H\leq 2 n_2^{\ell-k}H.
$$
Now we rewrite $ \gamma(\mathbf m )$ as an indicator function of the following simultaneous statements: 
\begin{equation}\label{eq:event02}
	\Delta \mid (m_1-m_2),\quad
	|m_1-m_2| \leq | \Delta| H,\quad 
	|m_2n_1^{\ell-k}-m_1n_2^{\ell-k}| \leq  |\Delta| H.
\end{equation}
The latter two conditions can be merged as $m_1\in J(\mathbf n,m_2)$ for some interval $J(\mathbf n,m_2)$ whose length can be bounded above using the first property of $\mathcal{ M}$. This yields 
\begin{align}\label{eqq26}
	|J(\mathbf n ,m_2)|\leq 2 |\Delta| H/ n_2^{\ell-k } 
	\leq   2H x^{\ell-k}/n_2^{\ell-k}
	\leq   2  H x^{(\ell-k)\kappa/d}
	. 
\end{align} 
Hence, equation \eqref{eq:9equ10} can be rewritten again as
\begin{equation}\label{eqq14}
	\begin{split}
		\sum_{\mathbf n \in \mathcal{M}}\sum_{\substack{|m_2|\leq 2 n_2^{\ell-k}H}}(Hx^d)^{\varepsilon}
		\left|S_1\right|,
	\end{split}
\end{equation} 
where
$$
S_1=\sum_{\substack{m_1\in 
		I_1(\mathbf n) \cap J(\mathbf n,m_2)  
		\\m_1\equiv m_2\bmod{\Delta}}}F(n_1^{k}m_1+g(n_1)).
$$
We now need to distinguish if $\gcd(n_1^km_2+g(n_1),\Delta)$ is larger than $x^{1-\kappa-\varepsilon}$ or not.\\
First we consider the case where $\gcd(n_1^km_2+g(n_1),\Delta)>x^{1-\kappa-\varepsilon}$. By assumption $(1)$ of Corollary \ref{cor_general}, we have that the expression in \eqref{eqq14} is
\begin{equation}
	\ll H^{\varepsilon}\sum_{\mathbf n \in \mathcal{M}}\sum_{\substack{|m_2|\leq 2 n_2^{\ell-k}H\\\gcd(n_1^km_2+g(n_1),\Delta)>x^{1-\kappa-\varepsilon}}}
	\left|\sum_{\substack{m_1\in 
			I_1(\mathbf n) \cap J(\mathbf n,m_2)  
			\\m_1\equiv m_2\bmod{\Delta}}}(Hx^d)^{\varepsilon}\right|.
\end{equation} 
By the estimate of the sizes of $J(\mathbf n,m_2)$ and $|\Delta|$, we have that the inner sum is
$$
\ll H^{1+\varepsilon}x^{(l-k)(2\kappa/d-1)}+1\ll H^{1+\varepsilon}x^{(l-k)(2\kappa/d-1)},
$$
since $H>x^{\ell-k}$.\\
Hence, we get that the expression is 
\begin{equation}
	\ll H^{1+\varepsilon}x^{(l-k)(2\kappa/d-1)}
	\sum_{\mathbf n \in \mathcal{M}}\sum_{\substack{|m_2|\leq 2 n_2^{\ell-k}H\\\gcd(n_1^km_2+g(n_1),\Delta)>x^{1-\kappa-\varepsilon}}}
	1.
\end{equation}
Now by Lemma \ref{large_gcd}, we get that the sum over $m_2$ is
$$
\ll \frac{x^{\ell-k}Hx^{\kappa\ell/d}}{x^{1-\kappa-\varepsilon}}+\tau(\Delta) \ll \frac{x^{\ell-k}Hx^{\kappa\ell/d}}{x^{1-\kappa-\varepsilon}},
$$
since $$
\gcd(n_1^k,\Delta)=
\gcd(n_1^k,n_1^{\ell-k}-n_2^{\ell-k})\ll 
\gcd(n_1,n_2)^{\ell}
\ll x^{\kappa\ell/d}.
$$
Combining everything, we get that \eqref{eqq14} is
$$
\ll x^2 H^{1+\varepsilon}x^{(\ell-k)(2\kappa/d-1)} \frac{x^{\ell-k}Hx^{\kappa\ell/d}}{x^{1-\kappa - \varepsilon}}
\ll \frac{x^2 H^{2}}{x^{1-\varepsilon-\kappa((3\ell-2k)/d+1)}}.
$$
Now we consider the case where $\gcd(n_1^km_2+g(n_1),\Delta)\leq x^{1-\kappa-\varepsilon}$.\\
Making the change of variables $m=n_1^{k}m_1+g(n_1)$ in \eqref{eqq14}, we see that 
$$
S_1=\sum_{\substack{m\in J'(\mathbf n,m_2)
		\\
		m\equiv u \bmod{q}}}F(m)    ,
$$ 
where 
\begin{itemize}
	\item $u$ is the unique solution $\bmod$ $q$ to $u\equiv n_1^{k}m_2+g(n_1) \bmod{\Delta}$ and $u\equiv g(n_1)\bmod{n_1^k}$, where $q=  \textnormal{lcm}(n_1^k, |\Delta|)\leq x^{\ell}$;
	\item $J'(\mathbf n,m_2)=n_1^{k}J(\mathbf n,m_2)+g(n_1)$  
	is an interval with
	\begin{align}\label{eq:jbound}
		|J'(\mathbf n,m_2) | =  n_1^{k} |J(\mathbf n ,m_2)|\leq 2n_1^k|\Delta|H/n_2^{\ell-k},
	\end{align}
	where the inequality follows from \eqref{eqq26}.
\end{itemize}
Also, using the second property of \eqref{eq:define_M}, we get that 
$$q\geq \frac{n_1^k\left|\Delta\right|}{\gcd(n_1^k,\Delta)}\geq\frac{n_1^k\left|\Delta\right|}{x^{\kappa\ell/d}}.$$
Furthermore, we have
\begin{align*}
	\gcd(u,q)&\leq \gcd(g(n_1),n_1^k)\gcd(n_1km_2+g(n_1),|\Delta|)\\
	&\leq x^{\kappa+\varepsilon}x^{1-\kappa-\varepsilon}\\
	&\leq x
\end{align*}
by the third property in \eqref{eq:define_M}.\\
We can now estimate the size of $S_1$ using assumption $(2)$ of Corollary \ref{cor_general}. Hence, we get that
\[
|S_1|
\ll \frac{ |J'(\mathbf n,m_2)|}{q x^{\delta}}  
\ll  \frac{ n_1^k |\Delta| H/n_2^{\ell-k}}{q x^{\delta}} 
\ll \frac{   H x^{\kappa\ell/d}/n_2^{\ell-k}}{ x^{\delta}} 
\ll  \frac{   H }{x^{\delta-\kappa\ell/d+(\ell-k)(1-\kappa)}} 
,  \]
except when $|J'(\mathbf n,m_2)|\leq H^{1-\varepsilon}x^k$; that case will be considered below.
Therefore, we can see by assumption $(1)$ that the left-hand side of~\eqref{eqq14} is 
\begin{align*}  
	&\ll \sum_{n_1\leq x}\sum_{n_2\leq x} \sum_{|m_2|\leq 2n_2^{\ell-k}H}(Hx^d)^{\varepsilon}
	\frac{   H }{x^{\delta-\kappa\ell/d+(\ell-k)(1-\kappa)}}\\
	&\ll \sum_{n_1\leq x}\sum_{n_2\leq x}\frac{   H^{2+\varepsilon }x^{l-k}}{x^{\delta-\kappa\ell/d+(\ell-k)(1-\kappa)}}\\
	&\ll \frac{x^2H^2}{x^{\delta-(\ell/d+\ell-k)\kappa-\varepsilon}}.
\end{align*}
This is an upper bound in this case.\\
In the case where $|J'(\mathbf n,m_2)|\leq H^{1-\varepsilon}x^k$, we use
$|F(n)|\ll (Hx^d)^{\varepsilon} $, which gives us the estimate
$$
|S_1|\ll (Hx^d)^{\varepsilon}  \left(\frac{|J'(\mathbf n,m_2)|}{q}+1\right)\ll (Hx^d)^{\varepsilon}  \left(\frac{H^{1-\varepsilon}x^{k+\kappa\ell/d}}{x^{\ell - k}}\right),
$$
since $q\leq x^{\ell}< H^{1-\varepsilon}x^k$.\\
The case $|J'(\mathbf n,m_2)|\leq H^{1-\varepsilon}x^k$ remains to be investigated. 
First, we assume $n_1>n_2$. By the definition of $J'(\mathbf n,m_2)$, we have that $J'(\mathbf n,m_2)$ is the intersubsection of the following two intervals:
\begin{align*}
	n_1^k\cdot[m_2-|\Delta|H&, m_2+|\Delta|H]\\
	n_1^k\cdot[m_2\frac{n_1^{\ell-k}}{n_2^{\ell - k}}-\frac{|\Delta|H}{n_2^{\ell-k}}&, m_2\frac{n_1^{\ell-k}}{n_2^{\ell - k}}+\frac{|\Delta|H}{n_2^{\ell-k}}]
\end{align*}
We know by the assumption $n_1>n_2$ that if the interval has length smaller than $H^{1-\varepsilon}x^k$, then the following inequality is satisfied:
$$
|m_2|+|\Delta|H-\frac{H^{1-\varepsilon}x^k}{n_1^k}<|m_2|\frac{n_1^{\ell-k}}{n_2^{\ell - k}}-\frac{|\Delta|H}{n_2^{\ell-k}}.
$$
We rearrange and recall the definition of $\Delta$. This gives us that
$$
|m_2|>H(n_2^{\ell-k}+1)-\frac{H^{1-\varepsilon}x^kn_2^{\ell-k}}{|\Delta|n_1^k}.
$$
In the case $n_1<n_2$, we start with the analogous inequalities and get the same result.\\
By the definition of $m_2$, we also know that $|m_2|<H(n_2^{\ell-k}+1)$, 	hence we have that $|m_2|\in U(n_2)$, where $U(n_2):=[H(n_2^{\ell-k}+1)-\frac{H^{1-\varepsilon}x^kn_2^{\ell-k}}{|\Delta|n_1^k},H(n_2^{\ell-k}+1)]$.\\
Plugging everything into equation \eqref{eqq14} and using the size estimate of $|\Delta|$, we get
\begin{align*}
	&\ll \sum_{n_1\leq x}\sum_{n_2\leq x} \sum_{|m_2|\in U(n_2)}(Hx^d)^{\varepsilon}
	(Hx^d)^{\varepsilon} \left(\frac{H^{1-\varepsilon}x^{k+\kappa\ell/d}}{x^{\ell - k}n_1^k}\right)\\
	&\ll H^{2+\varepsilon} \sum_{n_1\leq x}\sum_{n_2\leq x} \frac{n_2^{\ell-k}x^{2k+\kappa\ell/d}}{|\Delta|x^{\ell - k}n_1^{2k}}\\
	&\ll x^{2+\varepsilon}H^{2} \frac{x^{\kappa\ell/d}}{|\Delta|}\\
	&\ll x^2H^{2} \frac{1}{x^{(\ell-k)(1-\kappa/d)-\kappa\ell/d}-\varepsilon}.
\end{align*}
This provides an upper bound in this case.
\subsubsection{Optimising $\kappa$}
Now we have in total four different upper bounds for the desired quantity. Their denominators are $x^{\kappa/d+\varepsilon}$, $x^{1-\varepsilon-\kappa((3\ell-2k)/d+1)}$, $x^{\delta-(\ell/d+l-k)\kappa-\varepsilon}$ and $x^{(\ell-k)(1-\kappa/d)-\kappa\ell/d - \varepsilon}$.\\
We now want to find the optimal $\kappa$ to maximise the exponent. By the assumption that $d>1$, we have the following inequality:
\begin{align*}
	(\ell-k)(1-\kappa/d)-\kappa\ell/d - \varepsilon &>1-\varepsilon-\kappa(\ell/d+2(\ell-k)+1).
\end{align*}
Now $\delta-(\ell/d+l-k)\kappa-\varepsilon$ and $1-\varepsilon-\kappa(\ell/d+2(\ell-k)+1)$ are decreasing in $\kappa$, while $\kappa/d+\varepsilon$ is increasing in $\kappa$. Since neither of the decreasing functions is strictly smaller than the other for all relevant values of $\delta$ and $\kappa$, the optimal $\kappa$ is the minimal solution to the following two equations:
\begin{itemize}
	\item $\delta-(\ell/d+l-k)\kappa-\varepsilon=\kappa/d+\varepsilon$;
	\item $1-\varepsilon-\kappa((3\ell-2k)/d+1)=\kappa/d+\varepsilon$.
\end{itemize}
The first equation gives us that $\kappa=\frac{\delta-\varepsilon}{(1+\ell)/d+(\ell-k)}$, the second one gives us that\newline $\kappa=\frac{1-\varepsilon}{(1+3\ell-2k)/d+1}$.\\
Therefore, we have that $\kappa=\min\left(\frac{\delta-\varepsilon}{(1+\ell)/d+(\ell-k)} ,\frac{1-\varepsilon}{(1+3\ell-2k)/d+1}\right)$ is the optimal choice, which completes the proof.

\section{Equidistribution of square-free numbers in short intervals}\label{Chap_3}
In this section, we will prove the following result, which will help us to verify that the second assumption of Corollary \ref{cor_general} is satisfied.

\begin{proposition}[Square-free numbers in short intervals and arithmetic progressions]\label{prop-sqfree}
	Let $q\geq1, a\geq0$, let $h=\gcd\left(a,q\right)$ and $q'h=q$. Further, let $0<y<x$ be any real numbers, then
	\begin{align}
		\sum_{\substack{x-y\leq n \leq x\\n\equiv a\bmod q}} \mu^2\left(n\right) &= \frac{6}{\pi^2}\frac{\mu^2\left(h\right)y}{q}\prod_{p\mid q'} \left(1-\frac{1}{p}\right)^{-1}\prod\limits_{p\mid q}\left(1+\frac{1}{p}\right)^{-1}\nonumber\\ &+O\left(h\left(\left(\frac{x}{q}\right)^{1/2}+q^{1/2+\varepsilon}\right)\right).
	\end{align}
\end{proposition}
We can rewrite the right-hand side above in terms of the solutions of the polynomial $f~\bmod$ $p^2$. Therefore, we get
$$
\frac{6\mu^2(h)y}{\pi^2q}\prod_{p\mid q'} \left(1-\frac{1}{p}\right)^{-1}\prod\limits_{p\mid q}\left(1+\frac{1}{p}\right)^{-1}=\frac{y}{q}
\prod_{p}\left(1-\frac{\rho_f(p^2)}{p^2}\right),
$$ 
where $f\left(n\right)=qn+a$ and $\rho_f(m)$ is defined to be the number of solutions of a polynomial $f(n)\bmod m$ for $1\leq n\leq m$, hence it agrees with the conjectured asymptotic behaviour.
\begin{proof}
	We want to estimate
	\begin{equation*}
		S:=\sum_{\substack{x-y\leq n \leq x\\n\equiv a\bmod q}} \mu^2\left(n\right).
	\end{equation*}
	Our aim is to reduce the problem to the case where $\gcd(a,q)=1$ and then apply a theorem by Hooley.\\ 
	Let $h=\gcd(a,q)$ and $n'h=n$, $a'h=a$ and $q'h=q$, hence $\gcd(a',q')=1$. We have
	\begin{equation*}
		S =\sum_{\substack{x-y\leq n \leq x\\n\equiv a\bmod q}} \mu^2\left(n\right) = \sum_{\substack{\frac{x-y}{h}\leq n' \leq \frac{x}{h}\\n'\equiv a'\bmod q'}} \mu^2\left(hn'\right).
	\end{equation*}
	If $\gcd(h,n')\neq 1$, we have $\mu^2(hn')=0$. Therefore, we can impose that $\gcd(h,n')=1$. We get that
	\begin{align*}
		S&= \sum_{\substack{\frac{x-y}{h}\leq n' \leq \frac{x}{h}\\n'\equiv a'\bmod q'\\\gcd(h,n')=1}} \mu^2(h)\mu^2(n')\\
		&= \sum_{\substack{\frac{x-y}{h}\leq n' \leq \frac{x}{h}\\n'\equiv a'\bmod q'\\\gcd(\tilde{h},n')=1}} \mu^2(h)\mu^2(n'),
	\end{align*}
	where $$\tilde{h}=\prod_{\substack{p\mid h\\p\nmid q'}}p,$$
	since the congruence condition already ensures that $\gcd(n',q')=1$.\\
	Now we rewrite $\gcd(\tilde{h},n')$ as a sum of congruences and apply the Chinese Remainder Theorem to the inner sum, since $\gcd(\tilde{h},q')=1$ by the construction of $\tilde{h}$. Therefore,
	\begin{align*}
		S&=\mu^2(h) \sum_{\substack{1\leq b\leq\tilde{h}\\\gcd(b,\tilde{h})=1}} \sum_{\substack{\frac{x-y}{h}\leq n' \leq \frac{x}{h}\\n'\equiv a'\bmod q'\\n'\equiv b\bmod \tilde{h}}} \mu^2(n')\\
		&=\mu^2(h) \sum_{\substack{1\leq b\leq\tilde{h}\\\gcd(b,\tilde{h})=1}} \sum_{\substack{\frac{x-y}{h}\leq n' \leq \frac{x}{h}\\n'\equiv \tilde{a}\bmod q'\tilde{h}}} \mu^2(n'),
	\end{align*}
	where $\tilde{a}$ is the unique solution mod $q'\tilde{h}$ to $\tilde{a}\equiv a'\bmod q'$ and $\tilde{a}\equiv b\bmod \tilde{h}$.\\
	The inner sum now has a form where we can apply Theorem 3 from Hooley \cite{hooley_1975}. Also, we notice that the remaining sum is defined to be $\varphi(\tilde{h})$. This gives that
	\begin{align*}
		S&=\mu^2(h)\frac{6}{\pi^2}\frac{y}{hq'\tilde{h}}
		\prod_{p\mid q'\tilde{h}} \left(1-\frac{1}{p^2}\right)^{-1}\sum_{\substack{1\leq b\leq\tilde{h}\\\gcd(b,\tilde{h})=1}}1 + O\left(\varphi(\tilde{h})\left(\left(\frac{x}{q'\tilde{h}}\right)^{1/2}+(q'\tilde{h})^{1/2+\varepsilon}\right)\right)\\
		&=\mu^2(h)\frac{6}{\pi^2}\frac{y}{hq'}\frac{\varphi(\tilde{h})}{\tilde{h}}\prod_{p\mid q}\left(1-\frac{1}{p^2}\right)^{-1} + O\left(\tilde{h}\left(\left(\frac{xh/\tilde{h}}{q}\right)^{1/2}+q^{1/2+\varepsilon}\right)\right)\\
		&= \mu^2(h)\frac{6}{\pi^2}
		\frac{y}{q}\prod_{\substack{p\mid h\\p\nmid q'}} \left(1-\frac{1}{p}\right)\prod\limits_{p\mid q}\left(1-\frac{1}{p^2}\right)^{-1} + O\left(h\left(\left(\frac{x}{q}\right)^{1/2}+q^{1/2+\varepsilon}\right)\right)\\
		&= \frac{6}{\pi^2}\frac{\mu^2\left(h\right)y}{q}\prod_{p\mid q'} \left(1-\frac{1}{p}\right)^{-1}\prod\limits_{p\mid q}\left(1+\frac{1}{p}\right)^{-1} + O\left(h\left(\left(\frac{x}{q}\right)^{1/2}+q^{1/2+\varepsilon}\right)\right).
	\end{align*}
\end{proof}

\section{An analogue for $\mu^2(\cdot)$}\label{Chap_4}
This subsection aims to find an analogue for $\mu^2(\cdot)$.\\
We define
\begin{equation*}
	k_D(n):=\sum_{\substack{d^2\mid n \\ d\leq D}}\mu\left(d\right).
\end{equation*}
We will need to understand the values of $k_D(n)$ in arithmetic progression for short intervals to verify assumption $(2)$ of Corollary \ref{cor_general}. Further, we also need to understand these values for general polynomials. The first one will be calculated via explicit manipulations to arrive at the expression given by the asymptotics of $\mu^2\left(n\right)$ in arithmetic progressions in short intervals. The latter one will be derived via a general approach for the first $x$ values of an arbitrary polynomial whose coefficients have no common divisors.
\subsection{Preliminary}
Before proving any statements regarding $k_D(\cdot)$, we will investigate a bound on the average number of solutions$~\bmod$ $d^2$, which is stated and proven in \cite{shparlinski_average_2013}. We will repeat the proof along the same lines.\\
\begin{lemma}[Shparlinski]\label{lm_Igor}
	Let f be a square-free polynomial such that the coefficients have no common factor. Then, if $D\leq H^{A}$ for some A, the following two statements hold:
	$$\sum_{\substack{d\leq D\\\mu^2(d)=1}}\rho_f(d^2)=O(DH^{\varepsilon});$$
	$$\sum_{\substack{d>D\\\mu^2(d)=1}}\frac{\rho_f(d^2)}{d^2}=O(D^{-1}H^{\varepsilon}).$$
\end{lemma}
\begin{proof}
	Let $d_f$ be the degree of the polynomial $f$. We clearly have $\rho_f(p)\leq d_f$, hence by Hensel lifting, we have
	$$
	\rho_f(p^2)\leq d_f,
	$$
	if $p$ does not divide the discriminant $\Delta_f$ of $f$, which is non-zero since $f$ is square-free. Also, we have $\Delta_f=H^{O(1)}$. If $p\mid \Delta_f$, we have the trivial bound $\rho_f(p^2)\leq d_fp$.\\
	Since $\rho_f$ is multiplicative, we have for square-free $d$ that
	$$
	\rho_f(d^2)=\prod_{p\mid d}\rho_f(p^2)=d_f^{\omega(d)}\gcd(d,\Delta_f),
	$$
	where $\omega(d)$ is the number of prime divisors of $d$. By a standard bound of $\omega(d)$ we get
	$$
	\rho_f(d^2)\leq d_f^{\varepsilon}\gcd(d,\Delta_f).
	$$
	Now we have
	$$
	\sum_{d\leq D}\gcd(d,\Delta_f) \leq\sum_{e\mid\Delta_f}e\sum_{\substack{d\leq D\\ e\mid d}}1\leq D\tau(\Delta_f)=D\Delta_f^{\varepsilon}.
	$$
	This gives the first estimate
	$$
	\sum_{\substack{d\leq D\\\mu^2(d)=1}}\rho_f(d^2)
	=D^{\varepsilon}\sum_{d\leq D}\gcd(d,\Delta_f)=O(DH^{\varepsilon}).
	$$
	Further, we have
	\begin{align*}
		\sum_{d>D}\frac{\gcd(d,\Delta_f)}{d^2} &\leq \sum_{e\mid\Delta_f}e\sum_{\substack{d>D\\e\mid d}}\frac{1}{d^2}\leq\sum_{e\mid\Delta_f}\frac{1}{e}\sum_{\substack{d>D\\e\mid d}}\frac{1}{(d/e)^2}\\
		&\leq\sum_{e\mid\Delta_f}\frac{1}{e}\min\{(e/D),1\}.
	\end{align*}
	Considering the two cases $e>D$ and $e\leq D$ separately, we see that each of the $\tau(\Delta_f)$ summands is $D^{-1+\varepsilon}$. Hence, using estimates for $\Delta_f$ and $\tau$, we get the desired result
	$$
	\sum_{d>D}\frac{\rho_f(d^2)}{d^2}=O(D^{-1}H^{\varepsilon}).
	$$
\end{proof}
\subsection{$k_D(\cdot)$ in short arithmetic progressions}
\begin{lemma}\label{sqfull2}
	Let $q\geq 1$, $a\geq0$ and let $h=h_1h_2=\gcd(a,q)$ be such that $h_1$ is square-free and $h_2$ is square-full. We assume further that each prime factor of $h_2$ is less than $D^2$. Then we have that
	\begin{equation}
		\sum_{\substack{x-y\leq n\leq x\\n\equiv a \bmod q}}k_D\left(n\right) = 
		\frac{6}{\pi^2}\frac{\mu^2\left(h\right)y}{q}\prod_{p\mid q'} \left(1-\frac{1}{p}\right)^{-1}\prod\limits_{p\mid q}\left(1+\frac{1}{p}\right)^{-1} + O\left(\frac{y}{q}\frac{\sqrt{h_2}}{D}H^{\varepsilon}+DH^{\varepsilon}\right).
	\end{equation}
\end{lemma}
\begin{proof}
	As before, we start by exchanging summation. Thus, we have
	\begin{align*}
		S:&=\sum_{\substack{x-y\leq n\leq x\\n\equiv a \bmod q}}k_D\left(n\right)\\
		&=
		\sum_{\substack{x-y\leq n\leq x\\n\equiv a \bmod q}}\quad
		\sum_{\substack{d^2\mid n \\ d\leq D}}\mu\left(d\right)\\
		&=
		\sum_{d\leq D}\mu\left(d\right)
		\sum_{\substack{x-y\leq n\leq x\\n\equiv a \bmod q\\d^2\mid n}} 1.
	\end{align*}
	Let $h=\gcd(a,q)$. Then there exist $n',a',q'\in\NN$ such that $n'h=n$, $a'h=a$ and $q'h=q$, with $\gcd(a',q')=1$. We have
	\begin{align*}
		S&=
		\sum_{d\leq D}\mu\left(d\right)
		\sum_{\substack{\frac{x-y}{h}\leq n'\leq \frac{x}{h}\\n'\equiv a' \bmod q'\\d^2\mid hn'}} 1\\
		&=
		\sum_{k\mid h}
		\sum_{\substack{d\leq D\\\gcd(d^2,h)= k}}\mu(d) \sum_{\substack{\frac{x-y}{h}\leq n'\leq \frac{x}{h}\\n'\equiv a' \bmod q'\\\frac{d^2}{k}\mid n'}} 1.
	\end{align*}
	Let $h=h_1h_2$ and $k=k_1k_2$, where $h_1$ and $k_1$ are square-free, and $h_2$ and $k_2$ are square-full. The decomposition into a square-free and a square-full part gives us that $\gcd(h_1, h_2)=\gcd(k_1,k_2)=1$. Then we have by $\gcd(h_1, h_2)=1$ that 
	$$\gcd(d^2,h_1)\cdot\gcd(d^2,h_2)=\gcd(d^2,h_1h_2)=k_1k_2.$$
	In particular, we have that $\gcd(d^2,h_1)=k_1$, and $\gcd(d^2,h_2)=k_2$, since every prime divisor of the first factor has multiplicity 1, and every prime divisor of the second factor has multiplicity at least 2. Therefore, we have
	$$
	S=
	\sum_{k_2\mid h_2}
	\sum_{k_1\mid h_1}
	\sum_{\substack{d\leq D \\
			\gcd(d,h_1)=k_1 \\
			\gcd(d^2,h_2)=k_2}}\mu(d) \sum_{\substack{\frac{x-y}{h}\leq n'\leq \frac{x}{h}\\n'\equiv a' \bmod q'\\\frac{d^2}{k_1k_2}\mid n'}} 1.
	$$
	We can assume that $d$ is square-free since otherwise we have $\mu^2(d)=0$. Hence, $k_2$ is a perfect square since it contains no cubic factors, say $\left(\tilde{k_2}\right)^2=k_2$.\\
	We have $\gcd(d,h_2)=\tilde{k_2}$ and $\tilde{k_2}\mid\rad(h_2)$ since $d$ is square-free.
	Hence, we can rewrite the equation above as
	$$
	S=
	\sum_{\tilde{k_2}\mid \rad(h_2)}
	\sum_{k_1\mid h_1}
	\sum_{\substack{d\leq D \\\gcd(d,h_1)=k_1 \\\gcd(d,h_2)=\tilde{k_2}}}\mu(d) \sum_{\substack{\frac{x-y}{h}\leq n'\leq \frac{x}{h}\\n'\equiv a' \bmod q'\\\frac{d^2}{k_1\left(\tilde{k_2}\right)^2}\mid n'}} 1.
	$$
	In order for the congruence to have any solutions, we need $\gcd(\frac{d^2}{k_1\left(\tilde{k_2}\right)^2},q')=1$. Therefore, there exist $n'',a''\in\NN$ such that $n'=\frac{d^2}{k_1\left(\tilde{k_2}\right)^2}n''$ and $a'\equiv\frac{d^2}{k_1\left(\tilde{k_2}\right)^2}a''\bmod q'$. Then
	$$
	S=
	\sum_{\tilde{k_2}\mid \rad(h_2)}
	\sum_{k_1\mid h_1}
	\sum_{\substack{d\leq D\\\gcd(d,h_1)=k_1\\\gcd(d,h_2)= \tilde{k_2}\\\gcd(\frac{d^2}{k_1\left(\tilde{k_2}\right)^2},q')=1}}\mu(d) \sum_{\substack{\frac{x-y}{h}\frac{k_1\left(\tilde{k_2}\right)^2}{d^2}\leq n''\leq \frac{x}{h}\frac{k_1\left(\tilde{k_2}\right)^2}{d^2}\\n''\equiv a'' \bmod q'}} 1.
	$$
	Using the trivial estimate for the inner sum, we get $$\frac{y}{q'h}\frac{k_1\left(\tilde{k_2}\right)^2}{d^2} + O(1)=\frac{y}{q}\frac{k_1\left(\tilde{k_2}\right)^2}{d^2} + O(1).$$
	Therefore, we have
	$$
	S=
	\frac{y}{q}
	\sum_{\tilde{k_2}\mid \rad(h_2)}\left(\tilde{k_2}\right)^2
	\sum_{k_1\mid h_1}k_1
	\sum_{\substack{d\leq D\\\gcd(d,h_1)=k_1\\\gcd(d,h_2)= \tilde{k_2}\\\gcd(\frac{d^2}{k_1\left(\tilde{k_2}\right)^2},q')=1}}\frac{\mu(d)}{d^2} + O(DH^{\varepsilon}).
	$$
	Let $d'$ be such that $d=k_1\tilde{k_2}d'$. Then
	$$
	S=
	\frac{y}{q}
	\sum_{\tilde{k_2}\mid \rad(h_2)}\mu(\tilde{k_2})
	\sum_{k_1\mid h_1}\frac{\mu(k_1)}{k_1}
	\sum_{\substack{d'\leq \frac{D}{k_1\tilde{k_2}}\\\gcd(d',h)=1\\\gcd(d'k_1,q')=1}}\frac{\mu(d')}{\left(d'\right)^2} + O(DH^{\varepsilon}),
	$$
	since $k_1$, $\tilde{k_2}$ and $d'$ are pairwise coprime by the assumption that $d$ is square-free.\\
	Extending the inner sum to an infinite series, we get by the standard estimate an error term of
	$$
	O\left(\frac{y}{q}
	\sum_{\tilde{k_2}\mid \rad(h_2)}
	\sum_{k_1\mid h_1}\frac{1}{k_1}
	\sum_{\substack{d'> \frac{D}{k_1\tilde{k_2}}\\\gcd(d',h)=1\\\gcd(d'k_1,q')=1}}\frac{|\mu(d')|}{\left(d'\right)^2}\right)
	=O\left(\frac{y}{q}\frac{\rad(h_2)}{D}H^{\varepsilon}\right).
	$$
	Hence, the sum is
	$$
	S=\frac{y}{q}
	\sum_{\tilde{k_2}\mid \rad(h_2)}\mu(\tilde{k_2})
	\sum_{k_1\mid h_1}\frac{\mu(k_1)}{k_1}
	\sum_{\substack{d'\\\gcd(d',h)=1\\\gcd(d'k_1,q')=1}}\frac{\mu(d')}{\left(d'\right)^2} + O\left(\frac{y}{q}\frac{\rad(h_2)}{D}H^{\varepsilon}+DH^{\varepsilon}\right).
	$$
	Since $\gcd(k_1,d')=1$, we have
	$$
	S=\frac{y}{q}
	\sum_{\tilde{k_2}\mid \rad(h_2)}\mu(\tilde{k_2})
	\sum_{\substack{k_1\mid h_1\\\gcd(k_1,q')=1}}\frac{\mu(k_1)}{k_1}
	\sum_{\substack{d'\\\gcd(d',h)=1\\\gcd(d',q')=1}}\frac{\mu(d')}{\left(d'\right)^2} + O\left(\frac{y}{q}\frac{\sqrt{h_2}}{D}H^{\varepsilon}+DH^{\varepsilon}\right),
	$$
	using that $h_2$ is square-full. Since all three sums are now independent of each other, we can evaluate each of them separately.\\
	First, we obtain
	$$
	\sum_{\tilde{k_2}\mid \rad(h_2)}\mu(\tilde{k_2})=\mu^2(h)=\begin{cases}
		1 & \text{if } \rad(h_2)=1\\
		0 & \text{if } \rad(h_2)\neq1,
	\end{cases}
	$$
	by recalling the definition of $h_2$ and the fact that the sum is $1$ if $h_2=1$ and $0$ otherwise.\\
	The middle sum is
	$$
	\sum_{\substack{k_1\mid h_1\\\gcd(k_1,q')=1}}\frac{\mu(k_1)}{k_1} =
	\prod_{\substack{p\mid h_1\\p\nmid q'}}\left(1-\frac{1}{p}\right).
	$$
	The innermost sum is
	$$
	\sum_{\substack{d'\\\gcd(d',h)=1\\\gcd(d',q')=1}}\frac{\mu(d')}{\left(d'\right)^2} =
	\frac{6}{\pi^2}\prod_{p\mid q} \left(1-\frac{1}{p^2}\right)^{-1}.
	$$
	Combining the values of these three sums, we have
	$$
	S=\mu^2(h)\frac{6}{\pi^2}\frac{y}{q}\prod_{\substack{p\mid h_1\\p\nmid q'}}\left(1-\frac{1}{p}\right)
	\prod_{p\mid q} \left(1-\frac{1}{p^2}\right)^{-1} + O\left(\frac{y}{q}\frac{\sqrt{h_2}}{D}H^{\varepsilon}+DH^{\varepsilon}\right).
	$$
	In the case $\mu^2(h)=1$, this is the desired result since $h$ is square-free in this case, i.e. $h_1=h$.\\
\end{proof}

\subsection{Summing $k_D(\cdot)$ over polynomial values}
First, let us recall $\rho_f\left(m\right)$. It is defined to be the number of solutions to $f\left(n\right)\equiv 0 \bmod m$, where $1\leq n\leq m$. Furthermore, we denote by $H$ the upper bound for the size of the coefficients of $f$. Additionally, we assume that the $\gcd$ over all the coefficients is 1 and that $f$ is of degree at most $d$.\\
\begin{lemma}\label{lem_functions}
	Using the assumptions and notation above, if f is square-free, we have
	\begin{equation}
		\sum_{n\leq x} k_D\left(f\left(n\right)\right) =x\prod_{p}\left(1-\frac{\rho_f\left(p^2\right)}{p^2}\right) + O\left(DH^{\varepsilon}+\frac{x}{D}H^{\varepsilon}\right).
	\end{equation}
	If f is not square-free, then we have
	\begin{equation}
		\sum_{n\leq x} k_D\left(f\left(n\right)\right) = O\left(x^{1+\varepsilon}H^{\varepsilon}+D\right).
	\end{equation}
\end{lemma}
\begin{proof}
	First, we assume that $f$ is square-free, i.e. its discriminant $\Delta_f$ is non-zero.\\
	Again, our first step in calculating the sum is to interchange the order of summation. Afterwards, we reformulate the divisibility condition in the sum to a congruence condition and use the definition and multiplicativity of $\rho_f\left(m\right)$ to get the following result:
	\begin{align*}
		\sum_{n\leq x} k_D\left(f\left(n\right)\right) &= \sum_{n\leq x} \sum_{\substack{d^2\mid f\left(n\right)\\d\leq D }} \mu\left(d\right)\\
		&= \sum_{d\leq D}\mu\left(d\right)\sum_{\substack{n\leq x\\d^2\mid f\left(n\right)}}1\\
		&= x\sum_{d\leq D}\frac{\mu\left(d\right)}{d^2}\rho_f\left(d^2\right) + O\left(DH^{\varepsilon}\right)\\
		&= x\prod_{p}\left(1-\frac{\rho_f\left(p^2\right)}{p^2}\right) + O\left(DH^{\varepsilon}+\frac{x}{D}H^{\varepsilon}\right).
	\end{align*}
	The error terms were first calculated in \cite{shparlinski_average_2013}, and can also be seen in Lemma \ref{lm_Igor} above.\\
	
	We now assume that $f$ is not square-free and obtain that 
	\begin{align*}
		\sum_{n\leq x}k_D(f(n)) &=\sum_{n\leq x}\sum_{\substack{d^2\mid f\left(n\right)\\d\leq D }}\mu(d)\\
		&\leq\sum_{\substack{n\leq x\\f(n)\neq 0}}\sum_{d\mid f(n)}1+\sum_{\substack{n\leq x\\f(n)= 0}}d\\
		&=O(x^{1+\varepsilon}H^{\varepsilon}+D).
	\end{align*}
	By the standard estimate for the divisor function and since $f$ can have at most $d$ roots, the proof is complete.
\end{proof}

\section{Proof of Theorem \ref{thm_aim}}\label{Chap_5}
\begin{proof}[Proof of Theorem \ref{thm_aim}]
	We want to use Corollary \ref{cor_general}.
	We need to show that all the necessary conditions for Corollary \ref{cor_general} hold with our choices of $F\left(n\right)=\mu^2\left(n\right)-k_D\left(n\right)$ and $\delta=2/5+\varepsilon$. (The latter is an arbitrary choice, and one can take $\delta$ to be any value as long as $\varepsilon<\delta<1/2$.)\\
	To verify assumption $(1)$ of Corollary \ref{cor_general}, let us compare $F(n)$ to the divisor function. If $n=1$, we have $F(1)=0$. For $n>1$, we have
	\begin{equation}
		\left|k_D\left(n\right)\right| =
		\left|\sum_{\substack{d^2\mid n \\ d\leq D}}\mu\left(d\right)\right|
		\leq \sum_{d\mid n}\left|\mu\left(d\right)\right|
		=\tau\left(n\right)\ll n^{\varepsilon}
	\end{equation}
	by the standard estimate of the divisor function. Hence, also $F(n)\ll n^{\varepsilon}$.\\
	Now we check assumption $(2)$. By Proposition \ref{prop-sqfree} and Lemma \ref{sqfull2} we get
	\begin{align*}
		\sum_{\substack{n\in I\\n\equiv u\bmod q}}F(n) &= 
		\sum_{\substack{n\in I\\n\equiv u\bmod q}}\left(\mu^2(n)-k_D(n)\right)\\
		&=
		\frac{6}{\pi^2}\frac{\mu^2\left(h\right)|I|}{q}
		\prod_{p\mid q'} \left(1-\frac{1}{p}\right)^{-1}
		\prod\limits_{p\mid q} \left(1+\frac{1}{p}\right)^{-1} + O\left(h\left(\left(\frac{x^dH}{q}\right)^{1/2}+q^{1/2+\varepsilon}\right)\right)\\
		&-\frac{6}{\pi^2} \frac{\mu^2\left(h\right)|I|}{q}\prod_{p\mid q'} \left(1-\frac{1}{p}\right)^{-1}
		\prod\limits_{p\mid q} \left(1+\frac{1}{p}\right)^{-1} + O\left(\frac{|I|}{q}\frac{\sqrt{h_2}}{D}H^{\varepsilon}+DH^{\varepsilon}\right)\\
		&= O\left(h\left(\frac{Hx^d}{q}\right)^{1/2}+\frac{|I|}{q}\frac{\sqrt{h_2}}{D}H^{\varepsilon}+DH^{\varepsilon}\right),
	\end{align*}
	by recalling the assumption on $q$ and $H$, i.e. $q^2<Hx^d$.\\
	Plugging this into the expression in assumption $(2)$, we get
	\begin{align*}
		&\max_{\substack{1\leq u\leq q\\
				\gcd(u,q)\leq x
		}}  
		\sup_{\substack{
				I  \textnormal{ interval } 
				\\
				|I|> H^{1-\varepsilon}x^k
				\\
				I\subset [-2H x^d, 2H x^d]
		}}
		\frac{q}{|I|}
		\left|\sum_{\substack{n\in I\\n\equiv u\bmod q}}F(n)\right|\\
		\ll &\max_{\substack{
				1\leq u\leq q\\
				\gcd(u,q)\leq x
		}}   
		\sup_{\substack{
				I  \textnormal{ interval } 
				\\
				|I|> H^{1-\varepsilon}h^k
				\\
				I\subset [-2H x^d, 2H x^d]
		}}
		\frac{q}{|I|} \left(h(u)\left(\frac{Hx^d}{q}\right)^{1/2}+\frac{|I|}{q}\frac{\sqrt{h_2(u)}}{D}H^{\varepsilon}+DH^{\varepsilon}\right)\\
		\ll &\max_{\substack{
				1\leq u\leq q\\
				\gcd(u,q)\leq x
		}} \left(\frac{h(u)x^{\frac{d+\ell}{2}}}{H^{\frac{1}{2}}x^k}+\frac{\sqrt{h_2(u)}}{D}+\frac{Dx^{\ell}}{Hx^k}\right)H^{\varepsilon}.
	\end{align*}
	To finish the proof that condition $(2)$ holds, we need to verify that the term above is $\ll x^{-2/5+\varepsilon}$.\\
	The first summand is $\ll x^{-2/5+\varepsilon}$ by the assumptions that $x^{d+l-2k+14/5+\varepsilon}\leq H$ and $\gcd(u,q)=h\leq x$.\\
	By the assumption $\gcd(u,q)\leq x$, we see that $\sqrt{h_2}\leq\sqrt{x}$, hence, by choosing $D=x^{9/10+\varepsilon}$, we get $\frac{h_2}{D}H^{\varepsilon}\ll x^{-2/5+\varepsilon}$.\\
	The last summand is $\ll x^{-2/5+\varepsilon}$ by the definitions of $H$ and $D$.
	Hence, we conclude that both assumptions are verified.\\
	Therefore, we get the following:
	\begin{align*}
		\sup_{x'\in [x/2,x]}\sum_{|a|,|b|\leq H}\left|\sum_{1\leq n\leq x'} \mu^2(an^k+bn^{\ell}+g(n)) - k_D(an^k+bn^{\ell}+g(n))\right|^2\ll H^2x^{2-\frac{2-\varepsilon}{5+5\ell+5d(\ell-k)}},
	\end{align*}
	since under the choices $k<\ell\leq d$ and $\delta=2/5-\varepsilon$, we have that
	$$\min\left(\frac{2/5-\varepsilon}{1+\ell+d(\ell-k)} ,\frac{1-\varepsilon}{1+3\ell-2k+d}\right)=\frac{2/5-\varepsilon}{1+\ell+d(\ell-k)}.
	$$
	By Lemma \ref{lem_functions}, we have for square-free polynomials that
	$$
	\left|\sum_{1\leq n\leq x}k_D(f(n))-c_fx\right|\ll DH^{\varepsilon}+\frac{x}{D}H^{\varepsilon}.
	$$
	We denote by $\mathcal{G}_g^{\#}(H)$ the subset of $\mathcal{G}_g(H)$ that contains only the square-free polynomials.
	Therefore, we can bound the contribution of square-free polynomials to be at most
	\begin{align*}
		&\sup_{x'\in [x/2,x]} \sum_{f\in\mathcal{G}_g^{\#}(H)}\left|\sum_{1\leq n\leq x'} \mu^2(f(n)) - c_fx'\right|^2\\
		\ll
		&\sup_{x'\in [x/2,x]}\sum_{f\in\mathcal{G}_g^{\#}(H)} \left|\sum_{1\leq n\leq x'} \mu^2(f(n)) - k_D(f(n))+k_D(f(n))- c_fx'\right|^2\\		
		\ll &\sup_{x'\in [x/2,x]}\sum_{f\in\mathcal{G}_g^{\#}(H)} \left|\sum_{1\leq n\leq x'} \mu^2(f(n)) - k_D(f(n))\right|^2 + 
		\sup_{x'\in [x/2,x]}\sum_{f\in\mathcal{G}_g^{\#}(H)}\left|\sum_{1\leq n\leq x'}k_D(f(n))- c_fx'\right|^2\\
		\ll &H^2x^{2-\frac{2-\varepsilon}{5+5\ell+5d(\ell-k)}},
	\end{align*}
	by recalling that $D=x^{9/10+\varepsilon}$.\\
	If we can show that there are at most $O(H)$ non-square-free polynomials in $\mathcal{G}_g(H)$, then we can extend the sum above to the whole set $\mathcal{G}_g(H)$.\\
	We consider the $O(H)$ polynomials $g(n,a)=g(n)+an^{\ell}=c_e(a)n^e+\ldots+c_0(a)$, and we can assume that $c_e(a)\neq0$ and $c_0(a)\neq0$ by the assumption that $g(0)\neq0$ and $\ell\geq1$. Also we note that $e\leq d$.
	We now use that a polynomial $f$ is not square-free if and only if $\Delta_f=0$, where $\Delta_f$ is defined as $\Delta_f=\frac{(-1)^{(n-1)n/2}}{c_d}$Res$(f,f')$ and Res$(f,f')$ is the resultant of $f$ and its derivative $f'$.\\
	Viewing the resultant as a polynomial in the $c_i$'s, we see that the coefficient of $c_e^{e-1}c_0^{e-1}$ is $\pm(e)^{e}$, and the coefficient of $c_i^{e}c_e^{e-1-i}c_0^{i-1}$ is $\pm(i)^i(e-i)^{e-i}$.\\
	Now for each of the $O(H)$ many $b$'s, we consider the discriminant of the polynomials as a function in the coefficient of $n^k$. If $k\geq e$, there is at most one polynomial with leading coefficient $0$, hence were are done in this case. Therefore, we can assume from now on that the leading coefficient is non-zero.\\
	We know from above that the resultant this polynomial is not constantly zero, hence has only at most $2d-2$ many solutions. Therefore, in total, there are at most $O(H)$ non-square-free polynomials, each contributing at most $O(x^{1+\varepsilon}+D)=O(x^{1+\varepsilon})$. Therefore, the contribution of all non-square-free polynomials together is negligible.\\
	Hence,
	$$
	\sup_{x'\in [x/2,x]} \sum_{f\in\mathcal{G}_g(H)}\left|\sum_{1\leq n\leq x'} \mu^2(f(n)) - c_fx'\right|^2
	\ll H^2x^{2-\frac{2-\varepsilon}{5+5\ell+5d(\ell-k)}}.
	$$
	Now we want to find an upper bound for the number of polynomials that fail to satisfy the predicted asymptotic behaviour. We achieve this by estimating the following set:
	$$
	E_{\eta}(x,H):=\#\left\{f\in\mathcal{G}_g(H):\left|\sum_{1\leq n\leq x}\mu^2(f(n))-c_fx\right|>\eta x\right\}.
	$$
	By equation \eqref{equ2}, we have
	\begin{align*}
		E_{\eta}(x,H)\leq\frac{1}{\eta^2x^2}\sum_{f\in\mathcal{G}_g(H)}\left|\sum_{1\leq n\leq x}\mu^2(f(n))-c_fx\right|^2\ll \frac{H^2}{\eta^2x^{\frac{2-\varepsilon}{5+5\ell+d(\ell-k)}}}.
	\end{align*}
	We take $\eta=x^{-\frac{1-\varepsilon}{10+10\ell+10d(\ell-k)}}$ and recall that $x\leq H^{(1-\varepsilon)/(d+l+14/5)}$ to complete the proof.
\end{proof}
\printbibliography

@misc{bhargava2014geometric,
		title={The geometric sieve and the density of squarefree values of invariant polynomials}, 
		author={Manjul Bhargava},
		year={2014},
		eprint={1402.0031},
		archivePrefix={arXiv},
		primaryClass={math.NT}
	}

@misc{bhargava2014positive,
		title={A positive proportion of plane cubics fail the Hasse principle}, 
		author={Manjul Bhargava},
		year={2014},
		eprint={1402.1131},
		archivePrefix={arXiv},
		primaryClass={math.NT}
	}

@misc{bhargava2022average,
		title={On average sizes of Selmer groups and ranks in families of elliptic curves having marked points}, 
		author={Manjul Bhargava and Wei Ho},
		year={2022},
		eprint={2207.03309},
		archivePrefix={arXiv},
		primaryClass={math.NT}
	}

@misc{bhargava2013ternary,
		title={Ternary cubic forms having bounded invariants, and the existence of a positive proportion of elliptic curves having rank 0}, 
		author={Manjul Bhargava and Arul Shankar},
		year={2013},
		eprint={1007.0052},
		archivePrefix={arXiv},
		primaryClass={math.NT}
	}

@misc{bhargava2013average4,
		title={The average number of elements in the 4-Selmer groups of elliptic curves is 7}, 
		author={Manjul Bhargava and Arul Shankar},
		year={2013},
		eprint={1312.7333},
		archivePrefix={arXiv},
		primaryClass={math.NT}
	}

@misc{bhargava2013average5,
		title={The average size of the 5-Selmer group of elliptic curves is 6, and the average rank is less than 1}, 
		author={Manjul Bhargava and Arul Shankar},
		year={2013},
		eprint={1312.7859},
		archivePrefix={arXiv},
		primaryClass={math.NT}
	}

@misc{bhargava2013binary,
		title={Binary quartic forms having bounded invariants, and the boundedness of the average rank of elliptic curves}, 
		author={Manjul Bhargava and Arul Shankar},
		year={2013},
		eprint={1006.1002},
		archivePrefix={arXiv},
		primaryClass={math.NT}
	}

@misc{browning_bateman-horn_2022,
		title = {Bateman-Horn, polynomial Chowla and the Hasse principle with probability 1},
		number = {{arXiv}:2212.10373},
		publisher = {{arXiv}},
		author = {Browning, Tim and Sofos, Efthymios and Teräväinen, Joni},
		urldate = {2023-03-20},
		date = {2022-12-20},
		eprinttype = {arxiv},
		eprint = {2212.10373 [math]},
	}

@misc{browning2023squarefree,
		title={Square-free values of random polynomials}, 
		author={Tim Browning and Igor Shparlinski},
		year={2023},
		eprint={2305.15493},
		archivePrefix={arXiv},
		primaryClass={math.NT}
	}

@misc{browning2022batemanhorn,
		title={Bateman-Horn, polynomial Chowla and the Hasse principle with probability 1}, 
		author={Tim Browning and Efthymios Sofos and Joni Teräväinen},
		year={2022},
		eprint={2212.10373},
		archivePrefix={arXiv},
		primaryClass={math.NT}
	}

@article{DESTAGNOL2019102794,
		title = {Rational points and prime values of polynomials in moderately many variables},
		journal = {Bulletin des Sciences Mathématiques},
		volume = {156},
		pages = {102794},
		year = {2019},
		issn = {0007-4497},
		author = {Kevin Destagnol and Efthymios Sofos},
	}

@article{Erdös_Cube,
		author = {Erdös, P.},
		title = {Arithmetical Properties of Polynomials},
		journal = {Journal of the London Mathematical Society},
		volume = {s1-28},
		number = {4},
		pages = {416-425},
		doi = {https://doi.org/10.1112/jlms/s1-28.4.416},
		year = {1953}
	}

@article{Granville1998ABCAU,
		title={ABC allows us to count squarefrees},
		author={Andrew Granville},
		journal={International Mathematics Research Notices},
		year={1998},
		volume={1998},
		pages={991-1009}
	}

@article{Hooley1967OnTP,
		title={On the power free values of polynomials},
		author={Chris A. Hooley},
		journal={Mathematika},
		year={1967},
		volume={14},
		pages={21-26}
	}

@article{hooley_1975,
		author = {Hooley, C.},
		title = {A Note on Square-Free Numbers in Arithmetic Progressions},
		journal = {Bulletin of the London Mathematical Society},
		volume = {7},
		number = {2},
		pages = {133-138},
		year = {1975}
	}

@article{Nunes_2017,
		doi = {10.1112/s0025579317000043},
		url = {https://doi.org/10.1112%2Fs0025579317000043},
		year = 2017,
		month = {jan},
		publisher = {Wiley},
		volume = {63},
		number = {2},
		pages = {483--498},
		author = {R. M. Nunes},
		title = {{ON} {THE} {LEAST} {SQUAREFREE} {NUMBER} {IN} {AN} {ARITHMETIC} {PROGRESSION}
		},
		
		journal = {Mathematika}
	}

@article{Poonen2002SquarefreeVO,
		title={Squarefree values of multivariable polynomials},
		author={Bjorn Poonen},
		journal={Duke Mathematical Journal},
		year={2002},
		volume={118},
		pages={353-373}
	}

@article{Reuss2013PowerfreeVO,
		title={Power‐free values of polynomials},
		author={T. Reuss},
		journal={Bulletin of the London Mathematical Society},
		year={2013},
		volume={47}
	}

@article{shparlinski_average_2013,
		title = {On the Average Number of Square-Free Values of Polynomials},
		volume = {56},
		pages = {844 -- 849},
		journaltitle = {Canadian Mathematical Bulletin},
		author = {Shparlinski, Igor E.},
		date = {2013},
	}

@article{Prachar1958berDK,
		title={{\"U}ber die kleinste quadratfreie Zahl einer arithmetischen Reihe},
		author={Karl Prachar},
		journal={Monatshefte f{\"u}r Mathematik},
		year={1958},
		volume={62},
		pages={173-176}
	}

@InCollection{Helfgott2007PowerfreeVR,
		Author = {Helfgott, Harald Andr{\'e}s},
		Title = {Power-free values, repulsion between points, differing beliefs and the existence of error.},
		BookTitle = {Anatomy of integers. Based on the CRM workshop, Montreal, Canada, March 13--17, 2006},
		ISBN = {978-0-8218-4406-9},
		Pages = {81--88},
		Year = {2008},
		Publisher = {Providence, RI: American Mathematical Society (AMS)},
	}

@article{landau_handbuch_1911,
		title = {Handbuch der Lehre von der Verteilung der Primzahlen},
		volume = {22},
		issn = {1436-5081},
		pages = {A26--A26},
		number = {1},
		journaltitle = {Monatshefte für Mathematik und Physik},
		shortjournal = {Monatshefte für Mathematik und Physik},
		author = {Landau, Edmund},
		date = {1911-12-01},
	}

@article{Ricci1933RicercheAS,
		title={Ricerche aritmetiche sui polinomi},
		author={Giovanni Ricci},
		journal={Rendiconti del Circolo Matematico di Palermo (1884-1940)},
		year={1933},
		volume={57},
		pages={433-475}
	}
\end{document}